\documentclass[12pt,reqno]{article}
\usepackage{amsmath,amsthm,mathrsfs,amsfonts,amssymb}
\usepackage{amstext}
\usepackage{mdwlist}

\RequirePackage{geometry}
\newtheorem{thm}{Theorem}[section]
\newtheorem{cor}[thm]{Corollary}
\newtheorem{lem}[thm]{Lemma}
\newtheorem{prop}[thm]{Proposition}
\newtheorem{defn}[thm]{Definition}

\numberwithin{equation}{section}

\def\A{\mathscr A}
\def\Ad{\mathscr A^\bullet}
\def\D{{\rm dist}}

\def\R{\mathcal{R}}

\begin{document}

\title{Perturbation analysis for the generalized inverses with prescribed idempotents in Banach algebras}

\author{Jianbing Cao\thanks{{\bf Email}: caocjb@163.com} \\
Department of Mathematics, East China Normal University,\\
Shanghai 200241, P.R. China\\
Department of mathematics, Henan Institute of Science and Technology\\
Xinxiang, Henan, 453003, P.R. China\\
\and
Yifeng Xue\thanks{{\bf Email}: yfxue@math.ecnu.edu.cn; Corresponding author}\\
Department of Mathematics, East China Normal University,\\
Shanghai 200241, P.R. China }
\date{}

\maketitle

\begin{abstract}
In this paper, we first study the perturbations and expressions for the generalized inverses $a^{(2)}_{p,q}$,
$a^{(1, 2)}_{p,q}$, $a^{(2, l)}_{p,q}$ and $a^{(l)}_{p,q}$ with prescribed idempotents  $p$ and $q$.
Then, we investigate the general perturbation analysis and error estimate for some of these generalized
 inverses when $p,\,q$ and $a$ also have some small perturbations.
\vspace{3mm}



\end{abstract}


\section{Introduction}

Let $\R$ be a unital ring and let $\R^\bullet$ denote the set of all idempotent elements in $\R$. Given
$p,\,q \in \R^\bullet$. Recall that an element $a\in\R$ has the $(p, q)$--outer generalized inverse
$b=a^{(2)}_{p,q}\in\R$ if $bab = b$, $ba = p$ and $1-ab = q$.
If $b=a^{(2)}_{p,q}$ also satisfies the equation $aba=a$, then we say $a$ has the $(p, q)$--generalized inverse $b$,
in this case, written $b=a^{(1,2)}_{p,q}$. If an outer generalized inverse with prescribed idempotents exists, it is
necessarily unique (cf. \cite{DW1}). According to this definition, obviously, we see that the Moore--Penrose inverses in a
$C^*$-algebra and (generalized) Drazin inverses in a Banach algebra can be expressed by some $(p, q)$--outer generalized
inverses (cf. \cite{DW1,CLZ1,CX1}).

Based on some results of Djordjevi\'c and Wei in \cite{DW1}, Ilic, Liu and Zhong gave some equivalent conditions for the
existence of the $(p, q)$--outer generalized inverse in a Banach algebra in \cite{CLZ1}. But in our recent paper
\cite{CX1}, we find that Theorem 1.4 of \cite{CLZ1} is wrong. In \cite{CX1}, we first present a counter--example to
\cite[Theorem 1.4]{CLZ1}, then based on our counter--example, we define a new type of generalized inverse with prescribed
idempotents in a Banach algebra as follows:
\begin{defn}[see \cite{CX1}]\label{mdef1.2}
Let $a \in \mathscr{A}$ and $p,\; q \in \mathscr{A}^\bullet$. An element $b \in\mathscr{A}$ satisfying
$$
bab = b, \quad R_r(b)= R_r(p), \quad K_r(b) = R_r(q),
$$
will be called the $(p, q, l)$--outer generalized inverse of $a$, written as $a^{(2,l)}_{p,q} = b.$

In addition, if $a^{(2,l)}_{p,q}$ satisfies $a=aa^{(2,l)}_{p,q}a$, we call $a^{(2,l)}_{p,q}$ is the $(p,q,l)$--generalized inverse of $a$, denoted by $a^{(l)}_{p,q}$.
\end{defn}

Perturbation analysis of the generalized inverses is very important in both theory and applications. In recent years,
there are many fruitful results concerning the perturbation analysis for various types generalized inverses of operators
on Hilbert spaces or Banach spaces. The concept of stable perturbation of an operator on Hilbert spaces and Banach spaces
is introduced by Chen and Xue in \cite{CX1}. Later the notation is generalized to the set of Banach algebras
by the second author in \cite{Xue2} and to the set of Hilbert $C^*$--modules by Xu, Wei and Gu in \cite{XWG}.
Using the notation ``stable perturbation", many important results in perturbation analyses for Moore--Penrose inverses on
Hilbert spaces and Drazin inverses on Banach spaces or in Banach algebras have been obtained. Please see
\cite{CWX1,C-X1,C-X2,Xue1,Xue2,XC2} for detail.

Let $X, Y$ be Banach spaces over complex field $\mathbb C$. Let $T$ (resp. $S)$ be a given closed space in $X$ (resp,
$Y$). Let $A$ be a bounded linear operator from $X$ to $Y$ such that $A_{T,S}^{(2)}$ exists. The perturbation analysis
of $A_{T,S}^{(2)}$ for small perturbation of $T$, $S$ and $A$ has been done in \cite{DX2,DX0}.
Motivated by some recent results concerning the perturbation analysis for the generalized inverses of operators,
in this paper, we mainly study the perturbations and expressions for various types of generalized inverses with prescribed
idempotents in Banach algebras. We first consider the stable perturbation characterizations for $a^{(2)}_{p,q}$,
$a^{(1, 2)}_{p,q}$, $a^{(2, l)}_{p,q}$ and $a^{(l)}_{p,q}$ with prescribed idempotents $p$ and $q$. Then, by using stable
perturbation characterizations, we can investigate the general
perturbation analysis and error estimate for some of these generalized inverses when $p,\,q$ and $a$ also have some
small perturbations. The results obtained in this paper extend and improve many recent results in this area.

\section{\bf Preliminaries}
In this section, we give some notations in this paper, we also list some preliminary results which will be frequently
used in our main sections. Throughout the paper, $\mathscr{A}$ is always a complex Banach algebra with the unit 1.

Let $a\in\A$. If there is $b\in\A$ such that $aba=a$ and $bab=b$, then $a$ is called to be generalized invertible and
$b$ is called the generalized inverse of $a$, denoted by $b=a^+$. Let $Gi(\A)$ denote the set of all generalized invertible
elements in $\A\backslash\{0\}$. Let $\A^\bullet$ denote the set of all idempotent elements in $\A $. If $a\in Gi(\A )$,
then $a^+a$ and $1-aa^+$ are all idempotent elements. For $a\in \A $, set
\begin{alignat*}{2}
K_r(a)& = \{x \in \A  \;|\; ax = 0 \},&\quad R_r(a) &= \{ax \;|\; x \in \A \};\\
K_l(a)& = \{x \in \A  \;|\; xa = 0 \},&\quad R_l(a) &= \{xa \;|\; x \in \A \}.
\end{alignat*}
Clearly, if $p \in \A ^\bullet$, then $\A $ has the direct sum decompositions:
$$\A =K_r(p)\dotplus R_r(p) \quad or \quad  \A =K_l(p)\dotplus R_l(p).$$

The following useful and well--known lemma can be easily proved.
\begin{lem}\label{mlem1.1}
Let $x\in \A $ and $p\in\A ^\bullet$. Then
\begin{enumerate*}
\item[\rm(1)]  $K_r(p)$ and $R_r(p)$ are all closed and $K_r(p) = R_r(1-p),\ R_r(p)\mathscr{A} \subset R_r(p);$
\item[\rm(2)]  $px = x$ if and only if $R_r(x) \subset R_r(p)$ or $K_l(p) \subset K_l(x);$
\item[\rm(3)] $xp = x$ if and only if $K_r(p) \subset K_r(x)$ or $R_l(x)\subset R_l(p).$
\end{enumerate*}
\end{lem}

We list some of the necessary and sufficient conditions for the existence of $a^{(2,l)}_{p,q}$ in the following lemma,
which will be frequently used in the paper. Here we should indicate that $a^{(2,l)}_{p,q}$ is unique if it exists. Please see \cite{CX1} for the proofs and more information.

\begin{lem}\label{mlem1.3}
Let $a \in \mathscr{A}$ and $p, q \in \mathscr{A}^\bullet$. Then the following statements are equivalent:
\begin{enumerate*}
\item [\rm(1)] $a^{(2,l)}_{p,\,q}$ exists;
\item [\rm(2)] There exists $b \in \A $ such that $bab = b$, $R_r(b)= R_r(p)$ and $K_r(b) = R_r(q)$;
\item [\rm(3)] $K_r(a)\cap R_r(p)=\{0\}$ and $\A =aR_r(p)\dotplus R_r(q)$;
\item [\rm(4)] There exists $b \in \A $ satisfying $b = pb$, $p = bap$, $b(1- q) = b$, $1 -q = (1- q)ab$;
\item [\rm(5)] $p\in R_l((1-q)ap)=\{x(1-q)ap \;|\; x\in \mathscr{A}\}$ and $1-q \in R_r((1-q)ap)$;
\item [\rm(6)] There exist some $s, t \in \mathscr{A}$ such that $p=t(1-q)ap$, $1-q=(1-q)aps$.
\end{enumerate*}
\end{lem}

The following lemma gives some equivalent conditions about the existence of $a^{(l)}_{p,\,q}$. See \cite{CX1}
for more information.

\begin{lem}\label{mlem1.4}
Let $a \in \mathscr{A}$ and $p, q \in \A^\bullet$. Then the following conditions are equivalent:
\begin{enumerate*}
\item [\rm{(1)}] $a^{(l)}_{p,\,q}$ exists, i.e.,there exists some $b \in \mathscr{A}$ such that
  $$
  aba = a, \quad bab =b,\quad R_r(b) = R_r(p),\quad K_r(b) = R_r(q),
  $$
\item [\rm{(2)}] $\mathscr{A}=R_r(a)\dotplus R_r(q)= K_r(a)\dotplus R_r(p),$
\item [\rm{(3)}] $\mathscr{A}=aR_r(p)\dotplus R_r(q),\; R_r(a)\cap R_r(q)=\{0\}, \; K_r(a)\cap R_r(p)=\{0\}.$
\end{enumerate*}
\end{lem}

Let $X$ be a complex Banach space. Let $M,\,N$ be two closed subspaces in $X$. Set
$$
\delta(M,N)=\begin{cases}\sup\{\D(x,N)\,\vert\,x\in M,\,\|x\|=1\},\quad &M\not=\{0\}\\ 0 \quad& M=\{0\}
\end{cases},
$$
where $\D(x,N)=\inf\{\|x-y\|\,\vert\,y\in N\}$. The gap $\hat\delta(M,N)$ of $M,\,N$ is given by
$\hat{\delta}(M,N)=\max\{\delta(M,N),\delta(N,M)\}$. For convenience, we list some properties about $\delta(M,N)$
and $\hat\delta(M,N)$ which come from \cite{TK} as follows.

\begin{prop}[\text{\cite{TK}}]\label{2P1}
Let $M,\,N$ be closed subspaces in a Banach space $X$. Then
\begin{enumerate*}
  \item[$(1)$] $\delta(M,N)=0$ if and only if $M\subset N$;
  \item[$(2)$] $\hat{\delta}(M,N)=0$ if and only if $M=N$;
  \item[$(3)$] $\hat{\delta}(M,N)=\hat{\delta}(N,M)$;
  \item[$(4)$] $0\leq \delta(M,N)\leq 1$, $0\leq \hat{\delta}(M,N)\leq 1$.
\end{enumerate*}
\end{prop}

\section{Stable perturbations for the (p, q)--generalized inverses}

Let $a\in Gi(\A)$ and let $\bar a=a+\delta a\in\A$. Recall from \cite{Xue1} that $\bar a$ is a stable perturbation of
$a$ if $R_r(\bar a)\cap K_r(a^+)=\{0\}$. Obviously, we can define the stable perturbation for various kind of generalized
inverses. In this section, we concern the stable perturbation problem for various types of $(p,q)$--generalized inverses
in a Banach algebra.
\begin{lem}[{\rm\cite[Lemma 2.2]{DX3}}]\label{mlem2.1}
Let $a,\,b \in \A $. If $1+ab$ is left invertible, then so is $1+ba$.
\end{lem}

\begin{lem}\label{mlem2.2}
Let $a, \delta a \in \A $ and $p, \,q \in \A ^\bullet$ such that $a^{(2,l)}_{p,\,q}$ exists. Put $\bar{a}=a+\delta a$. If $1+\delta aa^{(2,l)}_{p,\,q}$ is invertible, $w=a^{(2,l)}_{p,\,q}(1+\delta aa^{(2,l)}_{p,\,q})^{-1}$. Then $\bar{a}^{(2,l)}_{p,\,q}$ exists and $w=\bar{a}^{(2,l)}_{p,\,q}$ .
\end{lem}

\begin{proof}
We prove our result by showing that $waw = w, R_r(w)= R_r(p), K_r(w) = R_r(q)$. It is easy to check that
$$
w=a^{(2,l)}_{p,\,q}(1+\delta aa^{(2,l)}_{p,\,q})^{-1}= (1+a^{(2,l)}_{p,\,q}\delta a)^{-1}a^{(2,l)}_{p,\,q}.
$$
Then, by using these two equalities, we can show $R_r(w)= R_r(a^{(2,l)}_{p,\,q})=R_r(p)$ and
$K_r(w)= K_r(a^{(2,l)}_{p,\,q})=R_r(q)$. We can also compute
\begin{align*}
w\bar{a}w&=a^{(2,l)}_{p,\,q}(1+\delta aa^{(2,l)}_{p,\,q})^{-1}\bar{a}a^{(2,l)}_{p,\,q}(1+\delta aa^{(2,l)}_{p,\,q})^{-1}\\
&=a^{(2,l)}_{p,\,q}(1+\delta aa^{(2,l)}_{p,\,q})^{-1}[(aa^{(2,l)}_{p,\,q}-1)+(1+\delta a a^{(2,l)}_{p,\,q})](1+\delta aa^{(2,l)}_{p,\,q})^{-1}\\
&=a^{(2,l)}_{p,\,q}(1+\delta aa^{(2,l)}_{p,\,q})^{-1}(aa^{(2,l)}_{p,\,q}-1)(1+\delta aa^{(2,l)}_{p,\,q})^{-1}+w\\
&=w.
\end{align*}
By Definition \ref{mdef1.2} and the uniqueness of $a^{(2,l)}_{p,\,q}$, we see $\bar{a}^{(2,l)}_{p,\,q}$ exists and $w=\bar{a}^{(2,l)}_{p,\,q}$ .
\end{proof}

Obviously, from the proof of Lemma \ref{mlem2.2}, we see that if $a^{(2,l)}_{p,\,q}$ exists and
$1+a^{(2,l)}_{p,\,q}\delta a$ is invertible, set $v=(1+a^{(2,l)}_{p,\,q}\delta a)^{-1}a^{(2,l)}_{p,\,q}$,
then we also have $v=\bar{a}^{(2,l)}_{p,\,q}$. In order to prove the main results about the stable perturbation,
we need one more characterizations of the existence of $a^{(2,l)}_{p,\,q}$.

For an element $a \in \A $ and $p,\; q \in \A ^\bullet$. Let $R_a\colon\A  \to \A $ be the right multiplier on $\A $(i.e., $R_a(x)=xa$ for any $x\in \A$). Then it easy to see that $a^{(2, l)}_{p,q}$ exists in  $\A $ if and only if
$(R_a)^{(2)}_{\A (1-q), \A (1-p)}$ exists in the Banach algebra $B(\A)$. So from the equivalences of (1), (2) and (3) in Lemma \ref{mlem1.3}, dually, we can get the following equivalent conditions for the existence of $a^{(2,l)}_{p,q}$.

\begin{prop}\label{mprop2.3}
Let $a \in \A $ and $p,\; q \in \A ^\bullet$.  Then the following statements are equivalent:
\begin{enumerate*}
\item [\rm(1)] $a^{(2,l)}_{p,\,q}$ exists;
\item [\rm(2)] There exists $c \in \A $ such that $cac = c$, $R_l(c) = R_l(1-q)$ and $K_l(c)= R_l(1-p)$;
\item [\rm(3)] $K_l(a)\cap R_l(1-q)=\{0\}$ and $\A =R_l(1-q)a\dotplus R_l(1-p)$.
\end{enumerate*}
\end{prop}

\begin{proof}
$(1)\Leftrightarrow(2)$ Suppose that $a^{(2,l)}_{p,\,q}$ exists. Let $c=a^{(2,l)}_{p,\,q}$. Then from Definition \ref{mdef1.2}, we know that $cac=c$, and then $ca, \, ac \in \A^\bullet$, $R_r(ca)=R_r(c) =R_r(p)$, $K_r(ac)=K_r(c)= R_r(q)$. Thus, it follows from Lemma \ref{mlem1.1} that
\begin{eqnarray*}
cap = p,\quad pca = ca, \quad ac(1- q)=ac, \quad (1- q)ac=1-q.
\end{eqnarray*}
Then, by using Lemma \ref{mlem1.1} again, we have
\begin{eqnarray}\label{eqr1.1}
K_l(ca)\subset K_l(p)\subset K_l(ca),\quad R_l(ac)\subset R_l(1-q)\subset R_l(ac).
\end{eqnarray}
By using $cac=c$, we have $K_l(ca)=K_l(c)$ and $R_r(ac)=R_r(c)$. Thus from Eq. \eqref{eqr1.1} we see that (2) holds. If (2) holds, similarly, by using Definition \ref{mdef1.2} and Lemma \ref{mlem1.1}, we can obtain $a^{(2,l)}_{p,\,q}$ exists.

$(2)\Leftrightarrow (3)$ By our remark above this lemma, we see these hold simply from the equivalences of (2) and (3) in Lemma \ref{mlem1.3}. Note that we can also prove these equivalences directly by using the right multiplier $R_a$ on $\A$. Here we omit the detail.
\end{proof}

Now we can present one of our main results about the stable perturbation of the generalized inverse $a^{(2,l)}_{p,\,q}$.

\begin{thm}\label{mthm2.4}
Let $a, \, \delta a \in \A $ and $p, \,q \in \A ^\bullet$ such that $a^{(2,l)}_{p,\,q}$ exists. Put $\bar{a}=a+\delta a$. Then the following statements are equivalent:
\begin{enumerate*}
  \item[\rm(1)] $1+\delta aa^{(2,l)}_{p,\,q}$ is invertible;
    \item[\rm(2)] $1+a^{(2,l)}_{p,\,q}\delta a$ is invertible;
    \item[\rm(3)] $\bar{a}^{(2,l)}_{p,\,q}$ exists.
\end{enumerate*}
In this case, we have $\bar{a}^{(2,l)}_{p,\,q}=a^{(2,l)}_{p,\,q}(1+\delta aa^{(2,l)}_{p,\,q})^{-1}=(1+a^{(2,l)}_{p,\,q}\delta a)^{-1}a^{(2,l)}_{p,\,q}.$
\end{thm}

\begin{proof}
$(1)\Leftrightarrow (2)$ follows from the well-known spectral theory in Banach algebras.

$(2)\Rightarrow (3)$ We prove our result by using Lemma \ref{mlem1.3}. Let $x\in K_r(\bar{a})\cap R_r(p)=\{0\}$. Since $R_r(p)= R_r(a^{(2,l)}_{p,\,q})$, then there exists some $t \in \A$ satisfying $x = a^{(2)}_{p,\,q}t$ and $\bar{a}t = 0$. Thus we have
\begin{align*}
(1+a^{(2,l)}_{p,\,q}\delta a)a^{(2,l)}_{p,\,q}t &=a^{(2,l)}_{p,\,q}t+a^{(2,l)}_{p,\,q}\delta aa^{(2,l)}_{p,\,q}t\\&=a^{(2,l)}_{p,\,q}(a+\delta a)a^{(2,l)}_{p,\,q}t\\&=a\bar{a}t=0.
\end{align*}
Since $1+a^{(2,l)}_{p,\,q}\delta a$ is invertible, it follows that $x=a^{(2,l)}_{p,\,q}t=0$. Therefore,
\begin{eqnarray}\label{eqth2.1}
K_r(\bar{a})\cap R_r(p)=\{0\}
\end{eqnarray}
Let $s \in \bar{a}R_r(p)\cap R_r(q)$. Since $R_r(p)= R_r(a^{(2,l)}_{p,\,q})$ and $R_r(q)=K_r(a^{(2,l)}_{p,\,q})$, then there exists some $z \in \A$ such that $s = \bar{a}a^{(2,l)}_{p,\,q}z$ and $a^{(2,l)}_{p,\,q}s=0$. Similar to the proof of Eq.\,\eqref{eqth2.1}, we can get $s=a^{(2,l)}_{p,\,q}t=0$, i.e, $\bar{a}R_r(p)\cap R_r(q)=\{0\}$. Since $1+a^{(2,l)}_{p,\,q}\delta a$ is invertible, then for any $w \in \A$ there is some $v\in \A$ such that $a^{(2,l)}_{p,\,q}w=(1+ a^{(2,l)}_{p,\,q}\delta a)v$. From $\bar{a}=a+\delta a$, we have
$$(1-a^{(2,l)}_{p,\,q}a)v=a^{(2,l)}_{p,\,q}(w- \bar{a} v) \in R_r(a^{(2,l)}_{p,\,q}a)\cap K_r(a^{(2,l)}_{p,\,q}a)=\{0\}.$$
Thus, $w-\bar{a }v \in K_r(a^{(2,l)}_{p,\,q})$ and $v=a^{(2,l)}_{p,\,q}a v\in R_r(a^{(2,l)}_{p,\,q})$. Since for $w \in \A$, we also have $w=\bar{a }v +(w-\bar{a }v )\in \bar{a}R_r(p)\dotplus R_r(q).$ Thus, we have
\begin{eqnarray}\label{eqth2.2}
\A =\bar{a}R_r(p)\dotplus R_r(q).
\end{eqnarray}
Now, from Eqs. \eqref{eqth2.1} and \eqref{eqth2.2}, by using Lemma \ref{mlem1.3}, we see that $\bar{a}^{(2,l)}_{p,\,q}$ exists.

$(3)\Rightarrow (1)$ Suppose that $\bar{a}^{(2,l)}_{p,\,q}$ exists, we want to prove $1+\delta aa^{(2,l)}_{p,\,q}$ is both left and right invertible. Since $\bar{a}^{(2,l)}_{p,\,q}$ exists, then from Lemma \ref{mlem1.3}, $\A =\bar{a}R_r(p)\dotplus R_r(q)=\bar{a}R_r(a^{(2,l)}_{p,\,q})\dotplus K_r(a^{(2,l)}_{p,\,q})$. Thus, for any $x\in\A$, we can write $x=\bar{a}a^{(2,l)}_{p,\,q}t_1 +t_2$, where $t_1\in \A$ and $t_2\in K_r(a^{(2,l)}_{p,\,q})$. Set $s=aa^{(2,l)}_{p,\,q}t_1+t_2$, then
\begin{align*}
(1+\delta aa^{(2,l)}_{p,\,q})s&=(1+\delta aa^{(2,l)}_{p,\,q})(aa^{(2,l)}_{p,\,q}t_1+t_2)\\&=\bar{a}a^{(2,l)}_{p,\,q}t_1 +t_2=x.
\end{align*}
Since $x\in \A$ is arbitrary, let $x=1$, then we see that $1+\delta aa^{(2,l)}_{p,\,q}$ is right invertible. Now we prove that $1+\delta aa^{(2,l)}_{p,\,q}$ is also left invertible. In fact, from Proposition \ref{mprop2.3}, we also have $\A =R_l(1-q)\bar{a}\dotplus R_l(1-p)=R_l(a^{(2,l)}_{p,\,q})\bar{a}\dotplus K_l(a^{(2,l)}_{p,\,q})$ for $\bar{a}^{(2,l)}_{p,\,q}$ exists. Then for any $z\in \A$, we can write $z=s_1a^{(2,l)}_{p,\,q}\bar{a} +s_2$, where $s_1\in R_l(a^{(2,l)}_{p,\,q}a)$ and $s_2\in K_l(a^{(2,l)}_{p,\,q})$. Let $t=s_1+s_2$,
then we have
\begin{align*}
t(1+a^{(2,l)}_{p,\,q}\delta a)&=(s_1+s_2)(1+a^{(2,l)}_{p,\,q}\delta a)\\&=s_1+s_2+s_1a^{(2,l)}_{p,\,q}(\bar{a}-a)\\&=s_1a^{(2,l)}_{p,\,q}\bar{a} +s_2 +s_1(1-a^{(2,l)}_{p,\,q}a)\\&=z.
\end{align*}
Since $z\in \A$ is arbitrary, let $z=1$, then we get that $1+ a^{(2,l)}_{p,\,q}\delta a$ is left invertible. But from Lemma \ref{mlem2.1} we see $1+\delta aa^{(2,l)}_{p,\,q}$ is also left invertible. Thus, $1+\delta aa^{(2,l)}_{p,\,q}$ is invertible.

Now, from Lemma \ref{mlem2.2}, $\bar{a}^{(2,l)}_{p,\,q}=a^{(2,l)}_{p,\,q}(1+\delta aa^{(2,l)}_{p,\,q})^{-1}=(1+a^{(2,l)}_{p,\,q}\delta a)^{-1}a^{(2,l)}_{p,\,q}.$
This completes the proof.
\end{proof}

\begin{lem}\label{mlem2.6}
Let $a,\, \delta a \in \A $ and $p, \,q \in \A ^\bullet$ such that $a^{(2,l)}_{p,\,q}$ exists. If $1+a^{(2,l)}_{p,\,q}\delta a$ is invertible. Put $\bar{a}=a+\delta a$ and
$f=(1+ a^{(2,l)}_{p,\,q}\delta a)^{-1}(1-a^{(2,l)}_{p,\,q}a)$. Then
\begin{enumerate*}
  \item[\rm(1)] $f\in \A^\bullet$ with $K_r(\bar{a})\subset R_r(f)$;
  \item[\rm(2)] $K_r(\bar{a})= R_r(f)$ if and only if $R_r(\bar{a})\cap R_r(q)=\{0\}$.
 \end{enumerate*}
\end{lem}

\begin{proof}
(1) Since $(1-a^{(2,l)}_{p,\,q}a)(1+a^{(2,l)}_{p,\,q} \delta a)=1- a^{(2,l)}_{p,\,q}a$ and $1+a^{(2,l)}_{p,\,q}\delta a$
is invertible, we have $1-a^{(2,l)}_{p,\,q}a=(1-a^{(2,l)}_{p,\,q}a)(1+a^{(2,l)}_{p,\,q}\delta a)^{-1}.$
Thus,
$$
f^2=(1+ a^{(2,l)}_{p,\,q}\delta a)^{-1}(1-a^{(2,l)}_{p,\,q}a)(1+ a^{(2,l)}_{p,\,q}\delta a)^{-1}(1-a^{(2,l)}_{p,\,q}a)
= f.
$$

Now for any $x\in \A$, from $(1-aa^{(2,l)}_{p,\,q})x = (1+ a^{(2,l)}_{p,\,q}\delta a)x -a^{(2,l)}_{p,\,q}\bar{a}x$, we
have
\begin{equation}\label{equa2.1}
fx=(1+ a^{(2,l)}_{p,\,q}\delta a)^{-1}(1-a^{(2,l)}_{p,\,q}a)x=x-(1+ a^{(2,l)}_{p,\,q}\delta a)^{-1}a^{(2,l)}_{p,\,q}
\bar{a}x.
\end{equation}
Eq. \eqref{equa2.1} implies that $K_r(\bar{a})\subset R_r(f)$.

(2) $(\Rightarrow)$ Let $t\in R_r(\bar{a})\cap R_r(q)=\{0\}$. Since $R_r(q)=K_r(a^{(2,l)}_{p,\,q})$, then there is some $x\in\A$ such that $t=\bar{a}x$ and $a^{(2,l)}_{p,\,q}\bar{a}x=a^{(2,l)}_{p,\,q}t=0$. Thus, $x=fx$ by Eq. \eqref{equa2.1}. So, $x\in R_r(f)=K_r(\bar{a})$ and $t=\bar{a}x=0$, i.e., $R_r(\bar{a})\cap R_r(q)=\{0\}$.

$(\Leftarrow)$ Thanks to (1), we need only to prove $R_r(f)\subset K_r(\bar{a})$. Let $t\in R_r(f)$. Since
$f\in \A^\bullet$, we have $t=ft$. So by Eq. \eqref{equa2.1}, we get
$(1+ a^{(2,l)}_{p,\,q}\delta a)^{-1}a^{(2,l)}_{p,\,q}\bar{a}t=0$ and then $a^{(2,l)}_{p,\,q}\bar{a}t=0$. Hence
$\bar{a}t \in R_r(\bar{a})\cap R_r(q)=\{0\}$ and $K_r(\bar{a})= R_r(f)$.
\end{proof}

Similar to \cite[Proposition 2.2]{Xue2} or \cite[Theorem 2.4.7]{Xue1}, but by using some of our characterizations for $a^{(2,l)}_{p,\,q}$ and $a^{(l)}_{p,\,q}$, we can
obtain the following results about the stable perturbations for these two kinds of generalized inverses.
\begin{thm}\label{mthm2.7}
Let $a,\, \delta a \in \A $ and $p, \,q \in \A ^\bullet$ such that $a^{(2,l)}_{p,\,q}$ exists. Suppose that $1+a^{(2,l)}_{p,\,q}\delta a$ is invertible. Put $\bar{a}=a+\delta a$ and
$w=(1+ a^{(2,l)}_{p,\,q}\delta a)^{-1}a^{(2,l)}_{p,\,q}$. Then the following statements are equivalent:
\begin{enumerate*}
  \item[\rm(1)] $w=\bar{a}^{(l)}_{p,\,q}$, i.e., $\bar{a}^{(2,l)}_{p,\,q}= \bar{a}^{(l)}_{p,\,q}$;
  \item[\rm(2)] $R_r(\bar{a})\cap R_r(q)=\{0\}$, i.e., $\bar{a}$ is a stable perturbation of $a$;
  \item[\rm(3)] $\bar{a}(1+a^{(2,l)}_{p,\,q}\delta a)^{-1}(1-a^{(2,l)}_{p,\,q}a)=0$;
  \item[\rm(4)] $(1-aa^{(2,l)}_{p,\,q})(1+\delta a a^{(2,l)}_{p,\,q})^{-1}\bar{a}=0$.
 \end{enumerate*}
\end{thm}
\begin{proof}
The implication $(1)\Leftrightarrow (2)$ comes from Lemma \ref{mlem1.3}, Lemma \ref{mlem1.4} and Theorem \ref{mthm2.4}.
The implication $(2)\Leftrightarrow (3)$ comes from Lemma \ref{mlem2.6}.

$(3)\Leftrightarrow (4)$ we can compute in the following way,
\begin{align*}
\bar{a}(1+a^{(2,l)}_{p,\,q}\delta a)^{-1}(1-a^{(2,l)}_{p,\,q}a)&=\bar{a} (1+a^{(2,l)}_{p,\,q}\delta a)^{-1}[(1+a^{(2,l)}_{p,\,q}\delta a)-a^{(2,l)}_{p,\,q}\bar{a}]\\&=\bar{a}-\bar{a}a^{(2,l)}_{p,\,q}(1+\delta aa^{(2,l)}_{p,\,q})^{-1}\bar{a}\\&=\bar{a}-[(1+\delta a a^{(2,l)}_{p,\,q})-(1-aa^{(2,l)}_{p,\,q})](1+\delta a a^{(2,l)}_{p,\,q})^{-1}\bar{a}\\&=(1-aa^{(2,l)}_{p,\,q})(1+\delta a a^{(2,l)}_{p,\,q})^{-1}\bar{a}.
\end{align*}
This completes the proof.
\end{proof}

Furthermore, by using the above theorem, we have the following results.
\begin{cor}\label{mcor2.8}
Let $a,\, \delta a \in \A $ and $p, \,q \in \A ^\bullet$ such that $a^{(2,l)}_{p,\,q}$ exists. Put $\bar{a}=a+\delta a$. If $1+a^{(2,l)}_{p,\,q}\delta a$ is invertible. Then the following statements are equivalent:
\begin{enumerate*}
  \item[\rm(1)] $R_r(\bar{a})\cap R_r(q)=\{0\}$, i.e., $\bar{a}$ is stable perturbation of $a$;
  \item[\rm(2)] $(1+a^{(2,l)}_{p,\,q}\delta a)^{-1}K_r(a^{(2,l)}_{p,\,q}a)=K_r(\bar a)$;
  \item[\rm(3)]  $(1+\delta a a^{(2,l)}_{p,\,q})^{-1}R_r(\bar{a})=R_r(aa^{(2,l)}_{p,\,q})$.
 \end{enumerate*}
\end{cor}
\begin{proof}
Note that we have $K_r(a^{(2,l)}_{p,\,q}a)=R_r(1-a^{(2,l)}_{p,\,q}a)$ and
$R_r(aa^{(2,l)}_{p,\,q})=K_r(1-aa^{(2,l)}_{p,\,q})$. So we can get the assertions by using Theorem \ref{mthm2.7}.
\end{proof}

\begin{thm}\label{tmthm2.7}
Let $a,\, \delta a \in \A $ and $p, \,q \in \A ^\bullet$ such that $a^{(l)}_{p,\,q}$ exists. Put $\bar{a}=a+\delta a$. Then the following statements are equivalent:
\begin{enumerate*}
  \item[\rm(1)] $1+a^{(l)}_{p,\,q}\delta a$ is invertible, $R_r(\bar{a})=K_r(q)$ and  $\bar{a}^{(l)}_{p,\,q}=a^{(l)}_{p,\,q}(1+ \delta a a^{(l)}_{p,\,q})^{-1}$.
  \item[\rm(2)]  $R_r(\bar{a})\cap R_r(q)=\{0\}$, $K_r(\bar{a})\cap R_r(p)=\{0\}$ and $\bar{a}R_r(p)=K_r(q)$.
 \end{enumerate*}
\end{thm}

\begin{proof}
$(1)\Rightarrow (2)$ Suppose that (1) holds. Since $a^{(l)}_{p,\,q}$ exists, we obtain that $a^{(2,l)}_{p,\,q}$ exists and $a^{(2,l)}_{p,\,q} =a^{(l)}_{p,\,q}$. Thus, from our assumption, by using Theorem \ref{mthm2.7} and Lemma \ref{mlem1.4}, we have $$R_r(\bar{a})\cap R_r(q)=\{0\}, \quad K_r(\bar{a})\cap R_r(p)=\{0\}.$$

Now we need to show $\bar{a}R_r(p)=K_r(q)$. But since $R_r(\bar{a})=K_r(q)$, so we can prove our result by showing that $\bar{a}R_r(p)=R_r(\bar{a})$. Obviously, $\bar{a}R_r(p) \subset R_r(\bar{a})$. On the other hand, since $\bar{a}^{(l)}_{p,\,q}$ exists, then by Lemma $\ref{mlem1.4}$ again, we have $\A=\bar{a}R_r(p) \dotplus R_r(q)$. Now for any $x\in R_r(\bar{a})$, we can write $x=x_1+x_2$ with $x_1 \in \bar{a}R_r(p)$ and $x_2 \in R_r(q)$. From $\bar{a}R_r(p) \subset R_r(\bar{a})$, we get $x_1 \in R_r(\bar{a})$. Thus, $$x_2=x-x_1 \in R_r(\bar{a})\cap R_r(q)=\{0\}.$$ Therefore, $x_2=0$ and then $x=x_1\in \bar{a}R_r(p)$. Hence, $\bar{a}R_r(p)= R_r(\bar{a})=K_r(q)$.

$(2)\Rightarrow (1)$ Since $q \in \A ^\bullet$ and $\bar{a}R_r(p)=K_r(q)$, we can write $\A= K_r(q) \dotplus R_r(q)= \bar{a}R_r(p) \dotplus R_r(q)$. Note that $R_r(\bar{a})\cap R_r(q)=\{0\}$, $K_r(\bar{a})\cap R_r(p)=\{0\}$, then by using Lemma \ref{mlem1.4}, we get $a^{(l)}_{p,\,q}$ exists, then $a^{(2, l)}_{p,\,q}$ also exists and $a^{(2,l)}_{p,\,q}=a^{(l)}_{p,\,q}$. Thus, from Theorem \ref{mthm2.4}, we see $1+a^{(l)}_{p,\,q}\delta a$ is invertible and $\bar{a}^{(l)}_{p,\,q}=a^{(l)}_{p,\,q}(1+ \delta a a^{(l)}_{p,\,q})^{-1}$. Now, by using Lemma \ref{mlem1.4}, we can get $\A=\bar{a}R_r(p) \dotplus R_r(q)$. Similarly, as in $(1)\Rightarrow (2)$, by using $R_r(\bar{a})\cap R_r(q)=\{0\}$, we can show that $\bar{a}R_r(p)=R_r(\bar{a})$ and then $R_r(\bar{a})=K_r(q)$. This completes the proof.
\end{proof}

The first result in the following lemma has been proved for generalized inverse $a^+$ by the second author (see \cite[Proposition 2.5]{Xue2}). By using the same method, we can prove the following results for the general inverse $a^{(l)}_{p,\,q}$.

\begin{lem}[{\rm\cite[Proposition 2.5]{Xue2}}]\label{mlem2.10}
Let $a, \delta a \in \A $ and $p, \,q \in \A ^\bullet$ such that $a^{(l)}_{p,\,q}$ exists. Put $\bar{a}=a+\delta a$.
\begin{enumerate*}
  \item[\rm(1)] If $\delta(R_r(\bar{a}), R_r(a))<\|1-aa^{(l)}_{p,\,q} \|^{-1}$, then $R_r(\bar{a}) \cap R_r(q)=\{0\}$;
  \item[\rm(2)] If $\delta(K_r(\bar{a}), K_r(a))<\|a^{(l)}_{p,\,q} a\|^{-1}$, then $K_r(\bar{a}) \cap R_r(p)=\{0\}$.
\end{enumerate*}
\end{lem}

By using Theorem \ref{mthm2.7}, Theorem \ref{tmthm2.7} and Lemma \ref{mlem2.10}, we have the following:
\begin{cor}\label{lemas1}
Let $a, \delta a \in \A $ and $p, \,q \in \A ^\bullet$ such that $a^{(l)}_{p,\,q}$ exists. Put $\bar{a}=a+\delta a$. If one of the following condition holds,
\begin{enumerate*}
  \item[\rm(i)] $\delta(K_r(\bar{a}), K_r(a))<\|a^{(l)}_{p,\,q} a\|^{-1}$, $\delta(R_r(\bar{a}), R_r(a))<\|1-aa^{(l)}_{p,\,q} \|^{-1}$ and $\bar{a}R_r(p)= K_r(q)$.
  \item[\rm(ii)] $1+a^{(l)}_{p,\,q}\delta a$ is invertible and $\delta(R_r(\bar{a}), R_r(a))<\|1-aa^{(l)}_{p,\,q} \|^{-1}$.
\end{enumerate*}
Then $\bar{a}^{(l)}_{p,\,q}$ exists and $\bar{a}^{(l)}_{p,\,q}=a^{(l)}_{p,\,q}(1+ \delta a a^{(l)}_{p,\,q})^{-1}$.
\end{cor}

Finally, we present some perturbation results for $a^{(2)}_{p,\,q}$.

\begin{thm}\label{mthm2.12}
Let $a,\, \delta a \in \A $ and $p, \,q \in \A ^\bullet$ such that $a^{(2)}_{p,\,q}$ exists. Put $\bar{a}=a+\delta a$. If $1+a^{(2)}_{p,\,q}\delta a$ is invertible. Then the following statements are equivalent:
\begin{enumerate*}
  \item[\rm(1)] $\bar{a}^{(2)}_{p,\,q}$ exists and $\bar{a}^{(2)}_{p,\,q}=a^{(2)}_{p,\,q}(1+\delta a a^{(2)}_{p,\,q} )^{-1}$;
  \item[\rm(2)] $\bar{a}p=(1-q)\bar{a}$;
  \item[\rm(3)]  $\bar{a}a^{(2)}_{p,\,q}=(1-q)\bar{a}a^{(2)}_{p,\,q}$ and $a^{(2)}_{p,\,q}\bar{a}=a^{(2)}_{p,\,q}\bar{a}p$.
 \end{enumerate*}
\end{thm}
\begin{proof}
$(1)\Leftrightarrow (2)$ comes from \cite[Theorem 4.1]{DW1}. We show that (2) and (3) are equivalent. If $\bar{a}p=(1-q)\bar{a}$, then $$\bar{a}a^{(2)}_{p,\,q}=\bar{a}pa^{(2)}_{p,\,q}=(1-q)\bar{a}a^{(2)}_{p,\,q}\quad and \quad a^{(2)}_{p,\,q}\bar{a}=a^{(2)}_{p,\,q}(1-q)\bar{a}=a^{(2)}_{p,\,q}\bar{a}p.$$
Conversely, if (3) holds, then $\bar{a}p=\bar{a}a^{(2)}_{p,\,q}a=(1-q)\bar{a}a^{(2)}_{p,\,q}a=aa^{(2)}_{p,\,q}\bar{a}p=(1-q)\bar{a}$.
\end{proof}

\begin{cor}
Let $a,\, \delta a \in \A $ and $p, \,q \in \A ^\bullet$ such that $a^{(2)}_{p,\,q}$ exists. Put $\bar{a}=a+\delta a$. If $1+a^{(2)}_{p,\,q}\delta a$ is invertible and $\delta a=(1-q)\delta a=\delta ap$. Then $\bar{a}^{(2)}_{p,\,q}$ exists and $$\bar{a}^{(2)}_{p,\,q}=a^{(2)}_{p,\,q}(1+\delta a a^{(2)}_{p,\,q} )^{-1}=(1+ a^{(2)}_{p,\,q}\delta a)^{-1}a^{(2)}_{p,\,q}.$$
\end{cor}

\begin{proof}
If $\delta a=(1-q)\delta a=\delta ap$, then it is easy to check that $\bar{a}p=(1-q)\bar{a}$. Thus, Theorem \ref{mthm2.12} shows that our results hold.
\end{proof}

\section{Perturbation analysis for the $(p, q)$--generalized inverses}

In this section, we mainly investigate the general perturbations problem for the $(p, q)$--generalized inverses
$a^{(2,l)}_{p,\,q}$ and $a^{(1,2)}_{p,\,q}$. Let $\kappa=\|a\|\|a^{(2,l)}_{p,\,q}\|$, which is the generalized
condition number of the generalized inverse $a^{(2,l)}_{p,\,q}$.

\begin{lem}\label{mlem3.1}
Let $a \in \mathscr{A}$ and $p \in \mathscr{A}^\bullet$ with $R_r(p)=R_r(a)$. Let $c \in \mathscr{A}$  with $R_r(c)$
closed and $\hat{\delta}(R_r(c),\, R_r(a))<\dfrac{1}{1+\|p\|}$. Then $\mathscr{A}=R_r(c)\dotplus  K_r(p)$.
\end{lem}

\begin{proof}
Let $L_p\,x=p\,x$, $\forall\,x\in\A$. Then $L_p$ is an idempotent operator on $\A$ with $\|L_p\|=\|p\|$ and
$\{L_p\,x\vert\,x\in\A\}=R_r(p)$. By \cite[Theorem 11]{VM} or \cite[Lemma 4.4.4]{Xue1}, $\hat{\delta}(R_r(c),\, R_r(a))<
\dfrac{1}{1+\|L_p\|}$ implies that $\A=R_r(c)\dotplus  K_r(p)$.
\end{proof}

\begin{lem}[{\cite[Lemma 2.4]{Xue2}}]\label{mlem3.2}
For any $p,\, q \in \mathscr{A}^\bullet$ we have $\hat{\delta}(R_r(p), R_r(q))\leq \|p-q\|$.
\end{lem}

\begin{lem}\label{mlem3.3}
Let $a \in \mathscr{A}$ and $p,\, q \in \mathscr{A}^\bullet$ such that $a^{(2,l)}_{p,\,q}$ exists. Suppose that $p^\prime \in\mathscr{A}^\bullet$ satisfying $\|p-p^\prime\|< \dfrac{1}{1+\kappa}$. Then
\begin{enumerate*}
\item[$(1)$] $\hat{\delta}(aR_r(p), aR_r(p^\prime))\leq \dfrac{\kappa\|p-p^\prime\|}{1-(1+\kappa)\|p-p^\prime\|}$;
\item[$(2)$] $aR_r(p^\prime) \subset \A$ is closed and $K_r(a) \cap R_r(p^\prime)=\{0\}$.
\end{enumerate*}
\end{lem}

\begin{proof}
(1) Set $b=a^{(2,l)}_{p,\,q}$. For any $t^\prime \in R_r(p^\prime)$, we have
\begin{align*}
{\text{dist}}(at^\prime,aR_r(p))&=\inf_{t\in R_r(p)}\|at^\prime-at\|\leq \|a\|\inf_{t\in R_r(p)}\|t-t^\prime\| \\&\leq \|a\|{\text{dist}}(t^\prime, R_r(p) )\\&\leq \|a\|\|t'\|\hat{\delta}(R_r(p^\prime),R_r(p)).
\end{align*}
Thus, we get
\begin{align}\label{eq2.2}
\delta(at^\prime, aR_r(p))\leq \|a\|\|t'\|\hat{\delta}(R_r(p^\prime),R_r(p)).
\end{align}

But for any $t^\prime \in R_r(p^\prime)$ and $t\in R_r(p)$, we have
\begin{align*}
\|b\|\|at^\prime\|&=\|b\|\|a(t^\prime-t+t)\|\geq \|b\|\|at\|-\|b\|\|a\|\|t^\prime-t\| \\
&\geq \|t\|-\|b\|\|a\|\|t^\prime-t\| \\
&\geq \|t^\prime\|-(1 +\|b\|\|a\|)\|t^\prime-t\|.
\end{align*}
Thus, $\|t^\prime\|- \|b\|\|at^\prime\| \leq (1 +\|b\|\|a\|)\|t^\prime-t\|,$ and then
\begin{align*}
\|t^\prime\|- \|b\|\|at^\prime\| & \leq (1 +\|b\|\|a\|){\text{dist}}(t^\prime, R_r(p))\\&\leq  (1 +\|b\|\|a\|)\|t'\|\hat{\delta}(R_r(p^\prime),R_r(p)).
\end{align*}

Therefore, we have
\begin{align}\label{eq2.3}
\|t'\| \leq \frac{\|b\|\|at^\prime\|}{1-(1 +\|b\|\|a\|)\hat{\delta}(R_r(p^\prime),R_r(p))}.
\end{align}

Then by Eq. \eqref{eq2.2} and Eq. \eqref{eq2.3}, we get
\begin{align*}
\delta(at^\prime, aR_r(p)) &\leq \frac{\|b\|\|a\| \|at^\prime\| \hat{\delta}(R_r(p^\prime),R_r(p))}{1-(1 +\|b\|\|a\|)\hat{\delta}(R_r(p^\prime),R_r(p))}.
\end{align*}

Now, by Lemma \ref{mlem3.2} and the Definition of gap--function, we have
\begin{align}\label{eq2.4}
\delta(aR_r(p^\prime), aR_r(p)) &\leq \frac{\kappa\|p-p^\prime\|}{1-(1 +\kappa)\|p-p^\prime\|}.
\end{align}

On the other hand, for any $t\in R_r(p)$, by Lemma \ref{mlem1.3}, we have $t=pt=bapt=bat,$ then
\begin{align*}
\D(at, aR_r(p^\prime))&=\inf_{s\in R_r(p^\prime)}\|at-as\| \leq \|a\|\inf_{s\in R_r(p^\prime)}\|t-s\| \\
&=\|a\|{\text{dist}}(t,R_r(p^\prime))\leq\|a\|\|t\|\delta(R_r(p),R_r(p^\prime))\\
&=\|a\|\|bat\|\delta(R_r(p),R_r(p^\prime))
\\&\leq\|a\|\|b\|\|at\|\delta(R_r(p),R_r(p^\prime)).
\end{align*}
Thus we have $\delta(aR_r(p), aR_r(p^\prime))\leq \kappa \hat{\delta}(R_r(p), R_r(p^\prime))$. So by Lemma \ref{mlem3.2},
\begin{align}\label{eq2.5}
\delta(aR_r(p), aR_r(p^\prime))\leq \kappa\|p-p^\prime\|.
\end{align}

Consequently, from Eq. \eqref{eq2.4} and Eq. \eqref{eq2.5},  we have
\begin{align*}
\hat{\delta}(aR_r(p), aR_r(p^\prime))&=\max\{\delta( aR_r(p^\prime), aR_r(p)), \delta(aR_r(p), aR_r(p^\prime))\}
\\&\leq\dfrac{\kappa\|p-p^\prime\|}{1-(1 +\kappa)\|p-p^\prime\|}.
\end{align*}

(2) Obviously, by Eq. \eqref{eq2.3}, we get $aR_r(p^\prime) \subset \A$ is closed and $K_r(a) \cap R_r(p^\prime)=\{0\}.$ This completes the proof.
\end{proof}

Now we can give the following perturbation result for $a^{(2,l)}_{p,\,q}$ when $p$ has a small perturbation.

\begin{thm}\label{mthm3.4}
Let $a\in \mathscr{A}$ and $p,\, q \in \mathscr{A}^\bullet$ such that $a^{(2,l)}_{p,\,q}$ exists. Suppose that
$p^\prime \in\mathscr{A}^\bullet$ with $\|p-p^\prime\|< \dfrac{1}{(1+\kappa)^2}$. Then $a_{p^\prime,q}^{(2,l)}$ exists
and
$$
\frac{\|a^{(2,l)}_{p^\prime, q} - a^{(2,l)}_{p,\,q}\|}{\| a^{(2,l)}_{p,\,q}\|}
\leq \dfrac{(1+ \kappa )\|p^\prime -p\|}{1-(1+ \kappa )\|p^\prime -p\|}\ \text{and}\
\|a^{(2,l)}_{p^\prime, q}\|\leq\frac{\|a^{(2,l)}_{p,\,q}\|}{1-(1+ \kappa )\|p^\prime -p\|}.
$$
\end{thm}
\begin{proof}
Let $b= a^{(2,l)}_{p,\,q}$, then by Definition \ref{mdef1.2}, we know that $ ab \in \Ad$ and $\kappa=\|b\|\|a\|\geq\|ab\|$. By using Lemma \ref{mlem3.3} and note that $\|p-p^\prime\|< \dfrac{1}{(1+\kappa)^2}$, we have
$$\hat{\delta}(aR_r(p), aR_r(p^\prime))\leq \dfrac{\kappa\|p-p^\prime\|}{1-(1+\kappa)\|p-p^\prime\|}< \frac{1}{1+\|ab\|}.$$ From Lemma \ref{mlem3.3}\,(2), we know that $K_r(a) \cap R_r(p^\prime)=\{0\}$ and $aR_r(p^\prime) \subset \A$ is closed.
Thus, by Lemma \ref{mlem3.1}, $aR_r(p^\prime)$ is complemented and $\mathscr{A}=aR_r(p^\prime) \dotplus R_r(q)$. Therefore, by Lemma \ref{mlem1.3}, we know $a^{(2)}_{p^\prime,\,q}$ exists.

For convenience we write $a^{(2,l)}_{p^\prime, q}=b^\prime$. Since we have proved that $a^{(2,l)}_{p^\prime, q}$ exists, then by Theorem \ref{mlem1.3},\,$\mathscr{A}=aR_r(p^\prime) \dotplus R_r(q)=aR_r(p^\prime) \dotplus K_r(b^\prime)$ and $K_r(b^\prime)=R_r(q)=K_r(b)$.  Thus for any $x\in \mathscr{A}$, we can write $x=t+ t^\prime$ with $t=ab^\prime z$ for some $z\in \mathscr{A}$ and $t^\prime\in R_r(q)$.

Since {\text{dist}}$(b^\prime z, R_r(p))\leq\|b^\prime z\|\delta(R_r(p^\prime),R_r(p))\leq |b^\prime z\|\|p^\prime -p\|$, then for every $\epsilon>0$, we can choose $y\in \mathscr{A}$ such that $\|b^\prime z-by \|<\|b^\prime z\|\|p^\prime -p\|+\epsilon.$ Put $s=aby$. Then we have
\begin{align}\label{eqprop2.5s}
\|t-s\|&=\|ab^\prime z - aby\|<\|a\|\|b^\prime z\||p^\prime -p\|+ \|a\|\epsilon.\notag \\
\|(b^\prime - b)x\|&=\|(b^\prime - b)(t+ t^\prime)\|=\|(b^\prime - b)t\|\notag \\
&\leq \|b^\prime ab^\prime  t- bab s\|+\| bs-bt\|\notag \\
&\leq \|b^\prime z- by\| + \|b\|\|s-t\|\notag \\
& \leq(1+ \kappa )\|b^\prime z\|\|p^\prime -p\|+(1+\kappa)\epsilon.
\end{align}
From $t=ab^\prime z$, we get $b^\prime t=b^\prime z$ and therefore
\begin{equation}\label{eqprop2.6s}
\|b^\prime z\|=\|b^\prime t\|=\|(b^\prime - b)x +bx\|\leq \|(b^\prime - b)x\|+\|b\|\|x\|.
\end{equation}
Thus by using Eq. \eqref{eqprop2.5s} and Eq. \eqref{eqprop2.6s}, we get
\begin{equation}\label{eqprop2.7s}
\|(b^\prime - b)x\|\leq (1+ \kappa )(\|b\|\|x\| +\|(b^\prime - b)x\|)\|p^\prime -p\|+(1+\kappa)\epsilon.
\end{equation}
Letting $\epsilon\to 0^+$ in Eq. \eqref{eqprop2.7s}, we can get
$$
\frac{\|a^{(2,l)}_{p^\prime, q} - a^{(2,l)}_{p,\,q}\|}{\| a^{(2,l)}_{p,\,q}\|}
\leq \dfrac{(1+ \kappa )\|p^\prime -p\|}{1-(1+ \kappa )\|p^\prime -p\|},\quad
\|a^{(2,l)}_{p^\prime, q}\|\leq\frac{\|a^{(2,l)}_{p,\,q}\|}{1-(1+ \kappa )\|p^\prime -p\|}.
$$
This completes the proof.
\end{proof}

Some representations for the generalized inverse $a^{(2,l)}_{p,q}$ have been presented in \cite{CX1}. The following
result gives a representation of $a^{(2,l)}_{p,q}$ based on (1,5) inverse. Note that this result is also an improvement
of the group inverses representation of $a^{(2,l)}_{p,q}$(see \cite{CX1}), which removes the existence of the group
inverses of $wa$ or $aw$.

\begin{lem}[{\cite[Theorem 5.6]{CX1}}]\label{mlem3.5}
Let $a, w \in \mathscr{A}$ and $p,\; q \in \mathscr{A}^\bullet$ such that $R_r(w) = R_r(p)$ and $K_r(w) = R_r(q)$. Then
the following statements are equivalent:
\begin{enumerate*}
\item [\rm(1)] $a^{(2,l)}_{p,q}$ exists;
\item [\rm(2)] $(aw)^{(1,5)}$ exists and $K_r(a) \cap R_r(w)= \{0\};$
\item [\rm(3)] $(wa)^{(1,5)}$ exists and $R_r(w)=R_r(wa).$
\end{enumerate*}
In this case, $waw$ is inner regular and $$a^{(2,l)}_{p,\,q}=(wa)^{(1,5)} w=w(aw)^{(1,5)}=w(waw)^{-}w.$$
\end{lem}

Now we give the result when $q$ has a small perturbation. By using our above lemma \ref{mlem3.5}, we can also give a new representation for the generalized inverse of the perturbed operator.

\begin{thm}\label{mthm3.6}
Let $a\in \mathscr{A}$ and $p, \,q \in \mathscr{A}^\bullet$ such that $a^{(2,l)}_{p,\,q}$ exists. Suppose that $q^\prime \in \Ad$ with $\|q-q^\prime\|< \dfrac{1}{2+\kappa}$. Then $a_{p,\,q\prime}^{(2,l)}$ exists and
\begin{enumerate*}
\item[$(1)$] $\dfrac{\|(a^{(2,l)}_{p,\,q^\prime}-a^{(2,l)}_{p,\,q})\|}{\|a^{(2,l)}_{p,\,q}\|}\leq
\dfrac{(1+\kappa)\|q-q^\prime\|}{1-\kappa\|q-q^\prime\|}$ and
$\|a^{(2,l)}_{p,\,q^\prime}\|\leq\dfrac{1+\|q'-q\|}{1-\kappa\|q'-q\|}\,\|a^{(2,l)}_{p,\,q}\|$.
\item[$(2)$] If there are some $w, \, v \in \mathscr{A}$ with $R_r(w) = R_r(p)$,\;$K_r(w) = K_r(v) =R_r(q)$ and
$R_r(v) = R_r(q^\prime)$. Then
$$
a_{p,\,q^\prime}^{(2,l)}=a^{(2,l)}_{p,\,q}+a^{(2,l)}_{p,\,q}(av)^{(1,5)}a(v-w)(1-aa^{(2,l)}_{p,\,q}).
$$
\end{enumerate*}
\end{thm}

\begin{proof}
Since $1+\|1-ab\|\leq 2+\kappa$, then by Lemma \ref{mlem3.2}, we have
$$
\hat{\delta}(R_r(q), R_r(q^\prime))\leq \|q-q^\prime\|<\frac{1}{2+\kappa}\leq \frac{1}{1+\|1-ab\|}.
$$
So $\mathscr{A}=aR_r(p) \dotplus R_r(q^\prime)$ by Lemma \ref{mlem3.1}. Note that $K_r(a) \cap R_r(p)=\{0\}$. So
$a^{(1,2)}_{p,\,q^\prime}$ exists.

(1) From $a^{(2,l)}_{p,\,q^\prime}=pa^{(2,l)}_{p,\,q^\prime}=a^{(2,l)}_{p,\,q}aa^{(2,l)}_{p,\,q^\prime}$, we get $$a^{(2,l)}_{p,\,q^\prime}-a^{(2,l)}_{p,\,q}=a^{(2,l)}_{p,\,q}(aa^{(2,l)}_{p,\,q^\prime}-aa^{(2,l)}_{p,\,q}).$$ Since $(1-aa^{(2,l)}_{p,\,q^\prime})x \in R_r(q^\prime)$ for any $x \in \A$,  so we have
\begin{align*}
\D((1-aa^{(2,l)}_{p,\,q^\prime})x, R_r(q))&\leq \hat{\delta}(R_r(q^\prime),R_r(q))\|(1-aa^{(2,l)}_{p,\,q^\prime})x\|\\&\leq \|q-q^\prime\|\|(1-aa^{(2,l)}_{p,\,q^\prime})x\|.
\end{align*}
Since $K_r(a^{(2,l)}_{p,\,q})=R_r(q)$, then, for any $\epsilon>0$, there is
$z\in \A$ such that $$\|(1-aa^{(2,l)}_{p,\,q^\prime})x-(1-aa^{(2,l)}_{p,\,q})z\| \leq\|q-q^\prime\|\|(1-aa^{(2,l)}_{p,\,q^\prime})x\|+\epsilon.$$ and then we have
\begin{align}\label{eq3.5}
\|(a^{(2,l)}_{p,\,q^\prime}-a^{(2,l)}_{p,\,q})x\|&=\|a^{(2,l)}_{p,\,q}(aa^{(2,l)}_{p,\,q^\prime}-aa^{(2,l)}_{p,\,q})x\|  \nonumber\\& \leq   \|a^{(2,l)}_{p,\,q}\{(1-aa^{(2,l)}_{p,\,q^\prime})x-(1-aa^{(2,l)}_{p,\,q})z\}\| \nonumber\\& \leq \left(\|q-q^\prime\|\|(1-aa^{(2,l)}_{p,\,q^\prime})x\|+\epsilon\right) \|a^{(2,l)}_{p,\,q}\|.
\end{align}
Since we also have
\begin{align}\label{eq3.6}
\|(1-aa^{(2,l)}_{p,\,q^\prime})x\|& \leq \|x\|+\|a\|\|a^{(2,l)}_{p,\,q}x-(a^{(2,l)}_{p,\,q^\prime}-a^{(2,l)}_{p,\,q})x\| \nonumber\\&\leq (1+\kappa)\|x\|+\|a\|\|(a^{(2,l)}_{p,\,q^\prime}-a^{(2,l)}_{p,\,q})x\|.
\end{align}
Now from Eqs. \eqref{eq3.5} and  \eqref{eq3.6}, we can compute $$\|(a^{(2,l)}_{p,\,q^\prime}-a^{(2,l)}_{p,\,q})x\|\leq \dfrac{(\|q-q^\prime\|(1+\kappa)\|x\|+\epsilon)\|a^{(2,l)}_{p,\,q}\|}{1-\kappa\|q-q^\prime\|}.$$
Let $\epsilon \to 0^+$ in the above inequality, we obtain that
$$
\frac{\|a^{(2,l)}_{p,\,q^\prime}-a^{(2,l)}_{p,\,q}\|}{\|a^{(2,l)}_{p,\,q}\|}\leq
\dfrac{(1+\kappa)\|q^\prime-q\|}{1-\kappa\|q^\prime-q\|},\quad
\|a^{(2,l)}_{p,\,q^\prime}\|\leq\frac{1+\|q'-q\|}{1-\kappa\|q'-q\|}\,\|a^{(2,l)}_{p,\,q}\|.
$$

(2) By Lemma \ref{mlem3.5}, $a^{(2,l)}_{p, q^\prime}=(va)^{(1,5)} v=v(av)^{(1,5)}$. For convenience, we write
$b=a^{(2,l)}_{p,\,q}$, $b^\prime=a_{p,\,q^\prime}^{(2,l)}$ and
$x=a^{(2,l)}_{p,\,q}+a^{(2,l)}_{p,\,q}(av)^{(1,5)}a(v-w)(1-aa^{(2,l)}_{p,\,q})$.
Now we prove that $x=b^\prime$ by using Lemma \ref{mlem1.3}\,(4). Obviously, we have $px=x$. Note that
$p=a^{(2,l)}_{p,\,q}ap$, so
\begin{align*}
xap&=\{a^{(2,l)}_{p,\,q}+a^{(2,l)}_{p,\,q}(av)^{(1,5)}a(v-w)(1-aa^{(2,l)}_{p,\,q})\}ap\\
&=a^{(2,l)}_{p,\,q}ap+a^{(2,l)}_{p,\,q}(av)^{(1,5)}a(v-w)(ap-ap)\\&=p.
\end{align*}
Since $b=w(aw)^{(1,5)}$. we have
\begin{align*}
x q^\prime &=\{a^{(2,l)}_{p,\,q}+a^{(2,l)}_{p,\,q}(av)^{(1,5)}a(v-w)(1-aa^{(2,l)}_{p,\,q})\}q^\prime\\
&=b q^\prime +\{b(av)^{(1,5)}av-b(av)^{(1,5)}avab-b(av)^{(1,5)}aw+b(av)^{(1,5)}awab   \}q^\prime.\\
&=b q^\prime +bab^\prime q^\prime-bab^\prime abq^\prime-b(av)^{(1,5)}awq^\prime+b(av)^{(1,5)}awabq^\prime.\\
&=b q^\prime +0-bq^\prime-b(av)^{(1,5)}awq^\prime+b(av)^{(1,5)}awaw(aw)^{(1,5)}q^\prime.\\
&=0.
\end{align*}
Thus, we have $x(1- q^\prime)=x$. Finally, since $aw= awaw(aw)^{(1,5)} =awab$, we have
\begin{align*}
(1- q^\prime)ax& =(1- q^\prime)a\{a^{(2,l)}_{p,\,q}+a^{(2,l)}_{p,\,q}(av)^{(1,5)}a(v-w)(1-aa^{(2,l)}_{p,\,q})\}\\
& =(1- q^\prime)ab^\prime a\{a^{(2,l)}_{p,\,q}+a^{(2,l)}_{p,\,q}(av)^{(1,5)}a(v-w)(1-aa^{(2,l)}_{p,\,q})\}\\
& =(1- q^\prime)\{ab^\prime ab+ab^\prime ab(av)^{(1,5)}(av-w)(1-ab)\}\\
& =(1- q^\prime)(ab+ab^\prime-ab-ab^\prime ab(av)^{(1,5)}aw+ab^\prime ab(av)^{(1,5)}awab)\\&=(1- q^\prime)ab^\prime\\&=1- q^\prime.
\end{align*}
Therefore by Lemma \ref{mlem1.3} and the uniqueness of $a^{(2,l)}_{p,\,q^\prime}$, we get that $x=a^{(2,l)}_{p,\,q^\prime}$.
\end{proof}

When the idempotents $p$ and $q$ both have some small perturbations, we have the following result.

\begin{thm}\label{mthm3.8}
Let $a\in \mathscr{A}$ and $p,\, q ,\, p^\prime,\, q^\prime \in \mathscr{A}^\bullet$ such that $a^{(2,l)}_{p,\,q}$ exists.
 If $\|p-p^\prime\|< \dfrac{1}{(1+\kappa)^2}$ and $\|q-q^\prime\|< \dfrac{1}{3+\kappa}$. Then
 $a_{p^\prime,\,q^\prime}^{(2,l)}$ exists and
\begin{align*}
\frac{\|a^{(2,l)}_{p^\prime,\,q^\prime}-a^{(2,l)}_{p,\,q}\|}{\|a_{p,\, q}^{(2,l)}\|}&\leq
\frac{(1+\kappa)(\|p-p'\|+\|q-q'\|)}{1-(1+\kappa)\|p-p'\|-\kappa\|q-q'\|}\\
\|a_{p^\prime,\,q^\prime}^{(2,l)}\|&\leq\dfrac{(1+\|q-q'\|)\|a_{p,\, q}^{(2,l)}\|}{1-(1+\kappa)\|p-p'\|-\kappa\|q-q'\|}.
\end{align*}
\end{thm}
\begin{proof}
By Theorem \ref{mthm3.4}, $a_{p^\prime,\,q}^{(2,l)}$ exists when $\|p-p^\prime\|< \dfrac{1}{(1+\kappa)^2}$ and in this case,
$$
\|a^{(2,l)}_{p^\prime, q}\|\leq\frac{\|a^{(2,l)}_{p,\,q}\|}{1-(1+ \kappa )\|p^\prime -p\|}\leq\frac{1+\kappa}{\kappa}
\|a^{(2,l)}_{p,\,q}\|.
$$
So $\|q-q^\prime\|< \dfrac{1}{3+\kappa}<\dfrac{1}{2+\|a\|\|a^{(2,l)}_{p',\,q}\|}$ and consequently,
$a^{(2,l)}_{p',\,q'}$ exists by Theorem \ref{mthm3.6}. Finally, by Theorem \ref{mthm3.4} and Theorem \ref{mthm3.6}, we have
\begin{align*}
\|a_{p^\prime,\,q^\prime}^{(2,l)}-a_{p,\,q}^{(2,l)}\|&\leq\|a_{p',\, q'}^{(2,l)}-a_{p',\, q}^{(2,l)}\|+
\|a_{p',\, q}^{(2,l)}-a^{(2,l)}_{p,\,q}\|\\
&\leq\frac{(1+\|a\|\|a_{p',\, q}^{(2,l)})\|q-q'\|}{1-\|a\|\|a_{p',\, q}^{(2,l)}\|\|q-q'\|}\|a_{p',\, q}^{(2,l)}\|
+\frac{(1+\kappa)\|p-p'\|}{1-(1+\kappa)\|p-p'\|}\|a_{p',\, q}^{(2,l)}\|\\
&\leq\frac{(1+\kappa)(1-\|p-p'\|)\|q-q'\|}{1-(1+\kappa)\|p-p'\|-\kappa\|q-q'\|}
\frac{\|a_{p',\, q}^{(2,l)}\|}{1-(1+\kappa)\|p-p'\|}\\
&\ +\frac{(1+\kappa)\|p-p'\|}{1-(1+\kappa)\|p-p'\|}\|a_{p',\, q}^{(2,l)}\|\\
&=\frac{(1+\kappa)(\|p-p'\|+\|q-q'\|)}{1-(1+\kappa)\|p-p'\|-\kappa\|q-q'\|}\|a_{p,\, q}^{(2,l)}\|
\end{align*}
and $\|a_{p^\prime,\,q^\prime}^{(2,l)}\|\leq\|a_{p^\prime,\,q^\prime}^{(2,l)}-a_{p,\,q}^{(2,l)}\|+\|a_{p,\, q}^{(2,l)}\|
\leq\dfrac{(1+\|q-q'\|)\|a_{p,\, q}^{(2,l)}\|}{1-(1+\kappa)\|p-p'\|-\kappa\|q-q'\|}.$
\end{proof}

Now, we consider the case when the elements $a, p, q\in \Ad$ all have some small perturbations.

\begin{thm}\label{bmthm3.9}
Let $a, \delta a\in \mathscr{A}$ and $p,\, q ,\, p^\prime,\, q^\prime \in \mathscr{A}^\bullet$ such that
$a^{(2,l)}_{p,\,q}$ exists. If $\|p-p^\prime\|< \dfrac{1}{(\kappa+1)^2}$, $\|q-q^\prime\|< \dfrac{1}{\kappa+3}$ and
$\|a^{(2,l)}_{p,\,q}\|\|\delta a\|<\dfrac{2\kappa}{(\kappa+1)(\kappa+4)}$. Then $\bar{a}_{p^\prime,\,q^\prime}^{(2,l)}$
exists and
\begin{align*}
\|\bar{a}_{p^\prime,\,q^\prime}^{(2,l)}\|&\leq
\frac{(1+\|q-q'\|)\|a_{p,\, q}^{(2,l)}\|}{1-(1+\kappa)\|p-p'\|-\kappa\|q-q'\|-(1+\|q-q'\|)\|a_{p,\, q}^{(2,l)}\|\|\delta a\|}\\
\|\bar{a}_{p^\prime,\,q^\prime}^{(2,l)}-a_{p,q}^{(2,l)}\|&\leq
\frac{\|a_{p,\,q}^{(2,l)}\|}{1-(1+\kappa)\|p-p'\|-\kappa\|q-q'\|}\Big[(1+\kappa)(\|p-p'\|+\|q-q'\|)\\
&\ +\frac{(1+\|q-q'\|)^2\|\delta a\|\|a_{p,\,q}^{(2,l)}\|}{1-(1+\kappa)\|p-p'\|-\kappa\|q-q'\|-(1+\|q-q'\|)
\|a_{p,\, q}^{(2,l)}\|\|\delta a\|}\Big].
\end{align*}
\end{thm}
\begin{proof}
Theorem \ref{mthm3.8} indicates that $a_{p^\prime,\,q^\prime}^{(2,l)}$ exists and
$$
\|a^{(2,l)}_{p'\, ,q'}\|\leq\dfrac{(1+\|q-q'\|)\|a_{p,\, q}^{(2,l)}\|}{1-(1+\kappa)\|p-p'\|-\kappa\|q-q'\|}
<\dfrac{(1+\kappa)(4+\kappa)}{2\kappa}\|a^{(2,l)}_{p,\,q}\|.
$$
Thus, $\|a^{(2,l)}_{p^\prime,q^\prime}\|\|\delta a\|<1$ and hence $1+a^{(2,l)}_{p^\prime,q^\prime}\delta a$ is invertible.
Therefore, $a^{(2,l)}_{p^\prime,q^\prime}$ exists and
$\bar{a}^{(2,l)}_{p^\prime,q^\prime}=a^{(2,l)}_{p^\prime,q^\prime}(1+\delta aa^{(2,l)}_{p^\prime,q^\prime})^{-1}$ by
Theorem \ref{mthm2.4}. Now by Theorem \ref{mthm3.8}, we have
$$
\|\bar{a}_{p^\prime,q^\prime}^{(2,l)}\|\!\leq\!
\dfrac{\|a_{p^\prime,q^\prime}^{(2,l)}\|}{1-\|a_{p^\prime,q^\prime}^{(2,l)}\|\|\delta a\|}
\!\leq\!\frac{[1+\|q-q'\|]\|a_{p,\, q}^{(2,l)}\|}{1-[1+\kappa]\|p-p'\|-\kappa\|q-q'\|-[1+\|q-q'\|]\|a_{p,\, q}^{(2,l)}\|\|\delta a\|}
$$
and
\begin{align*}
\|\bar{a}_{p^\prime,\,q^\prime}^{(2,l)}-a_{p,q}^{(2,l)}\|&\leq\|(1+{a}_{p^\prime,\,q^\prime}^{(2,l)} \delta a)^{-1}
{a}_{p^\prime,\,q^\prime}^{(2,l)}-a_{p',\,q'}^{(2,l)}\|+\|a_{p',\,q'}^{(2,l)}-a_{p,\,q}^{(2,l)}\|\\
&\leq\frac{\|\delta a\|\|a_{p^\prime,\,q^\prime}^{(2,l)}\|^2}{1-\|\delta a\|\|a_{p^\prime,\,q^\prime}^{(2,l)}\|}
+\|a_{p',\,q'}^{(2,l)}-a_{p,\,q}^{(2,l)}\|\\
&\leq\frac{\|a_{p,\,q}^{(2,l)}\|}{1-(1+\kappa)\|p-p'\|-\kappa\|q-q'\|}\Big[(1+\kappa)(\|p-p'\|+\|q-q'\|)\\
&\ +\frac{(1+\|q-q'\|)^2\|\delta a\|\|a_{p,\,q}^{(2,l)}\|}{1-(1+\kappa)\|p-p'\|-\kappa\|q-q'\|-(1+\|q-q'\|)
\|a_{p,\, q}^{(2,l)}\|\|\delta a\|}\Big].
\end{align*}
This completes the proof.
\end{proof}

By using perturbation theorems for the generalized inverse $a^{(2,l)}_{p,\,q}$, we can also investigate the perturbation
analysis for the generalized inverse $a^{(1,2)}_{p,\,q}$ under some conditions.

\begin{cor}\label{ncor3.11}
Let $a\in \mathscr{A}$ and $p,\, q \in \mathscr{A}^\bullet$ such that $a^{(1,2)}_{p,\,q}$ exists. Suppose that
$p^\prime \in\mathscr{A}^\bullet$ with $\|p-p^\prime\|< \dfrac{1}{(1+\kappa)^2}$ and $ap'=a$. Then $a_{p^\prime,\,q}^{(1,2)}$ exists
and
$$
\frac{\|a^{(1,2)}_{p^\prime, q} - a^{(1,2)}_{p,\,q}\|}{\| a^{(2,l)}_{p,\,q}\|}
\leq \dfrac{(1+ \kappa )\|p^\prime -p\|}{1-(1+ \kappa )\|p^\prime -p\|}\ \text{and}\
\|a^{(1,2)}_{p^\prime, q}\|\leq\frac{\|a^{(1,2)}_{p,\,q}\|}{1-(1+ \kappa )\|p^\prime -p\|}.
$$
\end{cor}

\begin{proof}Set $b=a^{(1,2)}_{p,\,q}$. Then $bab=b,\ aba=a,\ ba=p,\ 1-ab=q$ and $a^{(2,l)}_{p,\,q}=a^{(1,2)}_{p,\,q}=b$.
By Theorem \ref{mthm3.4}, $a^{(2,l)}_{p',\,q}$ is exists and
$$
\frac{\|a^{(2,l)}_{p^\prime, q} - a^{(1,2)}_{p,\,q}\|}{\| a^{(2,l)}_{p,\,q}\|}
\leq \dfrac{(1+ \kappa )\|p^\prime -p\|}{1-(1+ \kappa )\|p^\prime -p\|}\ \text{and}\
\|a^{(2,l)}_{p^\prime, q}\|\leq\frac{\|a^{(1,2)}_{p,\,q}\|}{1-(1+ \kappa )\|p^\prime -p\|}.
$$
We need only to show that $a^{(1,2)}_{p',\,q}=a^{(2,l)}_{p',\,q}$ in this case. Put $b'=a^{(2,l)}_{p',\,q}$. Then
$b'ab'=b'$, $(b'a)\A=p'\A$, $(1-ab')\A=q\A$. Thus, $(1-q)(1-ab')=0$ and hence $1-q=(1-q)ab'=abab'=ab'$. Furthermore,
$ab'a=(1-q)a=aba=a$. From $(b'a)\A=p'\A$, we get that $(1-b'a)p'=0$ and $p'=b'ap'=b'a$. Therefore,
$b'=a^{(1,2)}_{p',\,q}$.
\end{proof}

We need the following easy representation lemma for $a^{(1,2)}_{p,\,q}$.

\begin{lem}\label{flem3.11}
Let $a \in \mathscr{A}$ and $p,\, q \in \mathscr{A}^\bullet$ such that $a^{(1,2)}_{p,\,q}$ exists. Let $w \in \mathscr{A}$
such that $wa=p$ and $aw=1-q$. Then $a^{(1,2)}_{p,\,q}=(wa)^{\#} w=w(aw)^{\#}$.
\end{lem}

\begin{proof}
Obviously, $wa,\ aw\in\A^g$ for $wa=p$ and $aw=1-q$. We also have $(wa)^{\#}=p$ and $(aw)^{\#}=1-q$. Then by using the uniqueness of $a^{(1,2)}_{p,\,q}$, we can prove our lemma by simple computation.
\end{proof}

\begin{cor}\label{ncor3.12}
Let $a\in \mathscr{A}$ and $p,\, q \in \Ad$ such that $a^{(1,2)}_{p,\,q}$ exists. Suppose that $q^\prime\in \Ad$ with
$\|q-q^\prime\|< \dfrac{1}{2+\kappa}$ and $a=(1-q^\prime)a$. Then $a^{(1,2)}_{p, q^\prime}$ exists and
\begin{enumerate*}
\item[$(1)$] $\dfrac{\|(a^{(1,2)}_{p,\,q^\prime}-a^{(1,2)}_{p,\,q})\|}{\|a^{(1,2)}_{p,\,q}\|}\leq
\dfrac{(1+\kappa)\|q-q^\prime\|}{1-\kappa\|q-q^\prime\|}$ and
$\|a^{(1,2)}_{p,\,q^\prime}\|\leq\dfrac{1+\|q'-q\|}{1-\kappa\|q'-q\|}\,\|a^{(1,2)}_{p,\,q}\|$.
\item[$(2)$] If there are some $w, \, v \in \mathscr{A}$ with $wa = p=va$,\;$aw =1-q$ and $av = 1-q^\prime$.
Then
$$
a_{p,\,q^\prime}^{(1,2)}= a^{(1,2)}_{p,\,q}+a^{(1,2)}_{p,\,q}(av)^{\#}a(v-w)q.
$$
\end{enumerate*}
\end{cor}

\begin{proof}$a^{(2,l)}_{p,\,q^\prime}$ exists  by Theorem \ref{mthm3.6}. From $a=(1-q^\prime)a$ and Lemma \ref{mlem1.3},
we can obtain that $a^{(1.2)}_{p,\,q^\prime}$ exists and $a^{(2,l)}_{p,\,q^\prime}=a^{(1.2)}_{p,\,q^\prime}$.

Now the estimates in (1) and the representation for $a^{(1,2)}_{p,\,q'}$ in (2) follow from
Theorem \ref{mthm3.6} and Lemma \ref{flem3.11}.
\end{proof}

Finally, by Corollary \ref{ncor3.11}, Corollary \ref{ncor3.12} and Theorem \ref{mthm3.8}, we have

\begin{cor}\label{ncor3.13}
Let $a\in \mathscr{A}$ and $p,\, q ,\, p^\prime,\, q^\prime \in \mathscr{A}^\bullet$ with $a^{(1,2)}_{p,\,q}$ exists.
If  $\|p-p^\prime\|<\dfrac{1}{(1+\kappa)^2}$, $\|q-q^\prime\|<\dfrac{1}{3+\kappa}$ and
$ap^\prime=a=(1-q^\prime)a$. Then $a_{p^\prime,\,q^\prime}^{(1,2)}$ exists and
\begin{align*}
\frac{\|a^{(1,2)}_{p^\prime,\,q^\prime}-a^{(1,2)}_{p,\,q}\|}{\|a_{p,\, q}^{(1,2)}\|}&\leq
\frac{(1+\kappa)(\|p-p'\|+\|q-q'\|)}{1-(1+\kappa)\|p-p'\|-\kappa\|q-q'\|}\\
\|a_{p^\prime,\,q^\prime}^{(1,2)}\|&\leq\dfrac{(1+\|q-q'\|)\|a_{p,\, q}^{(1,2)}\|}{1-(1+\kappa)\|p-p'\|-\kappa\|q-q'\|}.
\end{align*}
\end{cor}

\end{document}